\documentclass[11pt,letterpaper]{amsart}
\usepackage[utf8]{inputenc}
\usepackage{array}
\usepackage{cite}
\usepackage{amssymb,amsthm,amsmath,mathrsfs,amsfonts}
\usepackage{fancyhdr}
\usepackage{amsmath,amssymb,amsthm}
\usepackage{amssymb, amsmath,amsthm,xcolor,bm,hyperref,color,tikz,euscript}
\usepackage{xcolor,bm,hyperref,tikz,euscript,sansmath,ytableau, mathtools}
\usetikzlibrary{3d,calc}
\usetikzlibrary{positioning,matrix,arrows,decorations.pathmorphing}
\usepackage{amssymb,amsthm,amsmath,mathrsfs}
\usepackage{graphicx,enumerate, microtype}
\usepackage{xcolor,bm,color,tikz,euscript,sansmath,ytableau}
\usepackage{tikz-cd}
\usepackage{subfig}
\usepackage{multirow}

\usepackage[all]{xy}

\usepackage{hyperref}
\usepackage{url, hypcap}
\definecolor{darkblue}{RGB}{0,0,160}
\hypersetup{
colorlinks,%
citecolor=darkblue,%
filecolor=black,%
linkcolor=darkblue,%
urlcolor=darkblue
}

\tikzstyle{invisivertex} = [black, shape=rectangle, minimum size=0pt, inner sep=3pt]
\tikzstyle{circledvertex} = [violet!80, draw, shape=circle, minimum size=0pt, inner sep=3pt]
\tikzstyle{dot} = [black, fill, shape=circle, minimum size=3pt, inner sep=2pt]
\tikzstyle{opendot} = [draw, black, shape=circle, minimum size=6pt, inner sep=1pt]
\tikzstyle{edgelabel} = [magenta, shape=circle, minimum size=6pt, inner sep=1pt]

\usepackage{cleveref}
\usepackage[all]{xy}

\usepackage[utf8]{inputenc}
\usepackage[english]{babel}
\usepackage[utf8]{inputenc}
\usepackage[english]{babel}
\usepackage{bbm}

\newcommand{\set}[2]{\left\{ #1 \;|\; #2 \right\}}

\newcommand{\ie}{{\textit{i.e.}}}

\newcommand{\cB}{\mathcal{B}}

\newcommand{\cL}{\mathcal{L}}

\newcommand{\Z}{\mathbb{Z}}

\newcommand{\cC}{\mathcal{C}}

\newcommand{\cF}{\mathcal{F}}

\newcommand{\RR}{\mathbb{R}}

\newcommand{\aaa}{\mathbf{a}}
\newcommand{\bbb}{\mathbf{b}}

\newcommand{\Poin}{\mathcal{\sf Poin}}

\newcommand{\Des}{{\sf Des}}

\newcommand{\StatMon}{{\sf u}}

\newcommand{\rank}{{\sf rank}}

\newcommand{\wt}{{\sf wt}}

\newcommand{\N}{\mathbb{N}}

\newcommand{\des}{\mathsf{des}}

\newcommand{\Inv}{\mathsf{Inv}}

\newcommand{\asc}{\mathsf{des}}

\newcommand{\extPsi}{{\text{\scriptsize\sf ex}\hspace*{-1pt}}\Psi}
\newcommand{\extPsiIota}{{\text{\scriptsize\sf ex}\hspace*{-1pt}}\widetilde\Psi}

\newcommand{\cfHP}{{\sf cfHP}}

\newcommand{\CHilb}{\underline{\CHilbAug}}
\newcommand{\CHilbAug}{\textsf{H}}
\newcommand{\redChi}{\overline{\Chi}}
\newcommand{\Chi}{\chi}

\newcommand{\Dfn}[1]{\emph{\bfseries #1}}
\newtheorem{lemma}{Lemma}

\newtheorem{proposition}[lemma]{Proposition}
\newtheorem{theorem}[lemma]{Theorem}
\newtheorem{corollary}[lemma]{Corollary}

\theoremstyle{definition}
\newtheorem{definition}[lemma]{Definition}
\newtheorem{to do}[lemma]{To Do}

\newtheorem{remark}[lemma]{Remark}
\newtheorem{question}[lemma]{Question}

\newtheorem{example}[lemma]{Example}

\numberwithin{lemma}{section}

\title[Chow polynomials as evaluations of extended $\aaa\bbb$-indices]{Chow and augmented Chow polynomials as evaluations of Poincaré-extended $\aaa\bbb$-indices}

\author[C.~Stump]{Christian Stump}
\address[C.~Stump]{Ruhr-Universit\"at Bochum, Germany}
\email{christian.stump@rub.de}


\begin{document}
\maketitle

\begin{abstract}
  We show that Chow polynomials and augmented Chow polynomials of matroids, and more generally of finite graded posets admitting $R$-labelings, are obtained as evaluations of their Poincaré-extended $\aaa\bbb$-indices.
  This implies in particular explicit combinatorial $\gamma$-positive expansions for both, providing the first proof of the $\gamma$-positivity not relying on the Kähler package for the Chow ring.
  We then evaluate this expansion to obtain an explicit closed formula for the braid arrangement.
\end{abstract}

\section{Main results}
\label{sec:main}

The Chow ring and the augmented Chow ring are important algebro-geometric objects attached to a loopless matroid that have been introduced in~\cite{MR2038195} and in~\cite{MR4477425}, respectively.
They have attracted a lot of attention in recent years, maybe most notable in their celebrated use in the proof of the log-concavity of the Whitney numbers of the first kind in~\cite{MR3862944}.
We follow~\cite{MR4749982} in our notations and call their Hilbert-Poincaré series~$\CHilb_M(x)$ and $\CHilbAug_M(x)$ the \Dfn{Chow polynomial} and the \Dfn{augmented Chow polynomial}.

\medskip

Recall that the \emph{lattice of flats}~$\cL(M)$ of a matroid~$M$ is geometric and thus admits an \Dfn{$R$-labeling}~$\lambda$, i.e., an edge-labeling~$\lambda$ such that each interval in~$\cL(M)$ contains a unique maximal chain along which~$\lambda$ is (weakly) increasing~\cite{MR0354473}.
For our purposes, we assume~$\lambda$ to have nonnegative integer values.
To a maximal chain $ \cF = \{\cF_0 \prec \dots \prec \cF_n\}$ in~$\cL(M)$ associate the sequence $\lambda_\cF = (\lambda_1,\dots,\lambda_n)$ of edge labels $\lambda_i = \lambda(\cF_{i-1} \prec \cF_i)$ for~$n = \rank(M)$.
We moreover set $\Des(\cF) = \{ i \mid \lambda_i > \lambda_{i+1}\}$ and $\des(\cF) = |\Des(\cF)|$, and we call $\Des(\cF)$ \Dfn{isolated}\footnote{This property has been called \emph{good} in~\cite{FerroniEtAl} and \emph{stable} in~\cite{brandenvecchi}.} if $i \in \Des(\cF)$ implies $i+1 \notin \Des(\cF)$.

\medskip

The following two theorems are the main results of this paper.

\begin{theorem}
\label{thm:main1}
  The Chow polynomial~$\CHilb_M(x)$ has the expansion
  \begin{align*}
     \CHilb_M(x)
     &= \sum_\cF x^{\des(\cF)} (x+1)^{n-1-2\ \!\!\des(\cF)}\,,
  \end{align*}
   where the sum ranges over all maximal chains~$\cF$ in~$\cL(M)$ such that $\Des(\cF)$ is isolated and $1 \notin \Des(\cF)$.
\end{theorem}

\begin{theorem}
\label{thm:main2}
  The augmented Chow polynomial~$\CHilbAug_M(x)$ has the expansion
  \begin{align*}
     \CHilbAug_M(x)
     &= \sum_\cF x^{\des(\cF)} (x+1)^{n-2\ \!\! \des(\cF)}\,,
  \end{align*}
   where the sum ranges over all maximal chains~$\cF$ in~$\cL(M)$ such that $\Des(\cF)$ is isolated.
\end{theorem}

\begin{corollary}
  The Chow polynomial~$\CHilb_M(x)$ and the augmented Chow polynomial~$\CHilbAug_M(x)$ are $\gamma$-positive.
  They are, in particular, nonnegative, palindromic and unimodal.
\end{corollary}

To the best of our knowledge, the only known proof of the $\gamma$-positivity appears in \cite{MR4749982} and relies on semi-small decompositions for the Chow and the augmented Chow rings~\cite{MR4477425} or, equivalently, on the Kähler package for these rings~\cite{MR3862944}.
Hence, \Cref{thm:main1,,thm:main2} may be viewed as a Hodge-theory free proof to the $\gamma$-positivity and also of unimodality of Chow and augmented Chow polynomials of geometric lattices.
We prove \Cref{thm:main1,,thm:main2} in \Cref{sec:evaluation}.
In \Cref{sec:typeA}, we then use those results to provide new explicit closed formulas for the (augmented) Chow polynomial in the case of the braid arrangement.

\begin{remark}[Generalization to finite graded posets admitting $R$-labelings]
  In work by Ferroni--Matherne--Vecchi, the definitions of the Chow and augmented Chow polynomials are generalized to finite graded posets with~$\hat{0}$ and~$\hat{1}$~\cite{FerroniEtAl}, see \Cref{def:Chow}.
  Using this definition for general finite graded posets, our proofs hold for all such posets admitting $R$-labelings.
  We, in particular, obtain non-negativity, palindromicity, unimodality, and $\gamma$-positivity in this level of generality.
\end{remark}

To prove the main results, we show for finite graded posets admitting $R$-labelings that the Chow and the augmented Chow polynomials are obtained as evaluations of the Poincaré-extended $\aaa\bbb$-index defined and studied by the author and Dorpalen-Barry--Maglione in~\cite{poincareextended}.
It turns out that the Chow polynomial, but not the augmented Chow polynomial, is an evaluation of the coarse flag Hilbert-Poincaré series as defined and studied by Maglione--Voll in~\cite{MaglioneVoll}, which itself is also an evaluation of the extended $\aaa\bbb$-index.
As the Poincaré-extended $\aaa\bbb$-indices of finite graded poset admitting $R$-labelings have explicit descriptions using the labeling along maximal chains, we then deduce~\Cref{thm:main1,,thm:main2}.

\begin{remark}[Three existing real-rootedness conjectures]
  In \Cref{thm:main3}, we show that the Chow polynomial is an evaluation of the extended $\aaa\bbb$-index.
  This evaluation is conjectured to be real-rooted~\cite[Conjecture~8.18]{ferroni2023valuativeinvariantslargeclasses}.
  It is shown in~\cite[Corollary~2.23]{poincareextended} that the numerator polynomial of the coarse flag Hilbert-Poincaré series is an evaluation of the extended $\aaa\bbb$-index and asked in~\cite[Question~1.6]{MR4614155} for which matroids a further evaluation is real-rooted.
  Also the chain polynomial of the lattice of flats is yet another evaluation of the extended $\aaa\bbb$-index which is conjectured in~\cite[Conjecture~1.2]{MR4565299} to be real-rooted.
  This suggests the following question.

  \begin{question}
    Which evaluations of the extended $\aaa\bbb$-index of a matroid, or of a more general finite graded poset admitting an $R$-labeling, are real-rooted?
  \end{question}
\end{remark}

\subsection*{Acknowledgment}

This work was conducted at the \emph{Mathematisches For\-schungs\-institut Oberwolfach} in June 2024 during the workshop ``Arrangements, Matroids and Logarithmic Vector Fields''.
We very much thank the institute for their hospitality and the organizers Takuro Abe, Graham Denham, Eva-Maria Feichtner, and Gerhard Röhrle for setting up this wonderful and inspiring meeting.
The results in this paper were obtained after attending the talk by Luis Ferroni entitled ``Chow polynomials of posets''.
The author would very much like to thank Luis and his collaborators Chris Eur, Jacob Matherne, and Lorenzo Vecchi, and as well Elena Hoster and Lukas Kühne for inspiring discussions concerning Chow polynomials and their combinatorial interpretations, and for comments on a preliminary version of this paper.
He also thanks three anonymous referees for their thoughtful comments that helped to improve the presentation of this paper.

\section{The Poincaré-extended $\aaa\bbb$-index and its evaluations}
\label{sec:evaluation}

Let~$P$ be a finite \Dfn{graded poset} of rank~$n$.
That is,~$P$ is a finite poset with unique minimum element~$\hat 0$ and unique maximum element~$\hat 1$ of rank~$n$ such that $\rank(F)$ is equal to the length of any maximal chain from~$\hat 0$ to~$F$.
Its \Dfn{Möbius function}~$\mu$ is given by $\mu(F, F) = 1$ for all $F\in P$ and $\mu(F,G) = -\sum_{F\leq H < G} \mu(F, H)$ for all $F < G$.
Its \Dfn{Poincaré polynomial} is
\begin{align*}
  \Poin_P(q)
  &= \sum_{F\in P} \mu(\hat{0}, F) \cdot (-q)^{\rank(F)}\,,
  \intertext{
    its \Dfn{characteristic polynomial} is
  }
  \Chi_P(q)
  &= \sum_{F\in P} \mu(\hat{0}, F) \cdot q^{n - \rank(F)}\,, \\
  \intertext{
    and its \Dfn{reduced characteristic polynomial} is
  }
  \redChi_P(q)
  &= \Chi_P(q) / (q-1) 
   = q^{n}\cdot \Poin_P(-q^{-1}) / (q-1)\,.
\end{align*}
A \Dfn{chain} $\cC = \{\cC_1 < \dots < \cC_k < \cC_{k+1}\}$ in~$P$ is an ordered set of pairwise comparable elements.
We assume throughout this paper that all chains end in the maximum element $\cC_{k+1} = \hat 1$.
A chain is \Dfn{maximal} if~$k = n$ and we usually write $\cF = \{\cF_0 \prec \dots \prec \cF_n\}$ in this case in order to clearly distinguish between general chains ending in~$\hat 1$ and maximal chains.
The \Dfn{chain Poincaré polynomial} and the \Dfn{reduced chain characteristic polynomial} of such a chain~$\cC$ are
\[
  \Poin_{P,\cC}(q)   =\prod_{i = 1}^k \Poin_{[\cC_i,\cC_{i+1}]}(q)
  \quad\text{ and }\quad
  \redChi_{P,\cC}(q) =\prod_{i = 1}^k \redChi_{[\cC_i,\cC_{i+1}]}(q)\,,
\]
where $[\cC_i,\cC_{i+1}] \subseteq P$ denotes the interval between two consecutive elements in the chain.
In particular,
\[
  \Poin_{P,\{\hat{0} < \hat 1\}}(q) = \Poin_P(q), \quad \redChi_{P,\{\hat{0} < \hat 1\}}(q) = \redChi_P(q)\,.
\]

\subsection{The Chow and the augmented Chow polynomials}

The following description of the Chow and augmented Chow polynomials of a matroid is taken from~\cite[Theorem~1.3 \& 1.4]{MR4749982}.
We present it for general finite graded posets as introduced by~\cite{FerroniEtAl}.

\begin{definition}[\!\!{\cite[Definition~3.3]{FerroniEtAl}}]
\label{def:Chow}
  Let~$P$ be a finite graded poset.
  The \Dfn{Chow polynomial} is defined to be~$1$ if $\rank(P) = 0$ and otherwise
  \[
    \CHilb_P(x)
      = \sum_{\cC} \redChi_{P,\cC}(x) \ \in \Z[x]\,,
  \]
    where the sum ranges over all chains $\cC = \{\cC_1 < \dots < \cC_{k+1}\}$ in~$P$ starting at $\cC_1 = \hat 0$ and ending in $\cC_{k+1} = \hat 1$.
  The \Dfn{augmented Chow polynomial} is defined to be~$1$ if $\rank(P) = 0$ and otherwise
    \[
      \CHilbAug_P(x)
        = \sum_{\cC} x^{\rank(\cC_1)}\cdot\redChi_{P,\cC}(x) \ \in \Z[x]\,,
    \]
    where the sum ranges over all chains $\cC = \{ \cC_1 < \dots < \cC_{k+1}\}$ in~$P$ ending in $\cC_{k+1} = \hat 1$.
\end{definition}

It is shown in \cite[Theorems~1.3 \& 1.4]{MR4749982} that for the lattice of flats $P = \cL(M)$, this definition agrees with the Hilbert-Poincaré series of the Chow and the augmented Chow rings for a loopless matroid~$M$ as presented in \Cref{sec:main}.

\subsection{The Poincaré-extended $\aaa\bbb$-index}
\label{sec:abindex}

The following description and properties of the extended $\aaa\bbb$-index are taken from~\cite{poincareextended}.
Throughout this section, let~$P$ be a finite graded poset admitting an \Dfn{$R$-labeling}, \ie, an  edge-labeling $\lambda : \mathcal E(P) \rightarrow \mathbb{N}_+$ such that each interval in~$P$ contains a unique maximal chain along which~$\lambda$ is (weakly) increasing.

\medskip

Let $\Z[y]\langle\aaa,\bbb\rangle$ be the polynomial ring in two noncommuting variables $\aaa,\bbb$ with coefficients being polynomials in the variable~$y$.
For a subset $S \subseteq \{0,\dots,n-1\}$, we set $\wt_S(\aaa,\bbb) = w_0\cdots w_{n-1}$ for the polynomial with
\begin{equation}
\label{eqn:ext-weights}
  w_k = \begin{cases}
          \bbb        & \text{if } k\in S, \\
          \aaa - \bbb & \text{if } k\notin S\,.
        \end{cases}
\end{equation}
For a chain $\cC = \{\cC_1 < \dots < \cC_k < \cC_{k+1}\}$ in~$P$, ending in $\cC_{k+1} = \hat 1$, we moreover set
\[
  \wt_\cC = \wt_{\{\rank(\cC_1),\dots,\rank(\cC_k)\}}\,.
\]

\begin{definition}
\label{def:extab}
  The (\Dfn{Poincaré-})\Dfn{extended $\aaa\bbb$-index} of~$P$ is
  \[
    \extPsi_P(y,\aaa,\bbb)
      = \sum_{\cC} \Poin_{P,\cC}(y) \cdot \wt_{\cC}(\aaa,\bbb) \ \in \Z[y]\langle\aaa,\bbb\rangle\,,
  \]
  where the sum ranges over all chains~$\cC = \{\cC_1 <\dots < \cC_{k+1}\}$ in~$P$ ending in $\cC_{k+1} = \hat 1$.
\end{definition}

Since~$P$ has a unique minimum, we always have $\Poin_P(0) = 1$, implying
\begin{equation}
\label{eq:abindex}
  \extPsi_P(0,\aaa,\bbb)
  = \Psi_P(\aaa,\bbb)
  = \sum_{\cC} \wt_\cC(\aaa,\bbb)
\end{equation}
for (a mild variation of) the \Dfn{$\aaa\bbb$-index} $\Psi_P(\aaa,\bbb)$ as given, for example, in~\cite[Section~2]{bayer-survey}.

\begin{proposition}[{\!\!\cite[Corollary~2.22]{poincareextended}}]
\label{prop:iotaapp}
  We have
  \[
    \bbb\cdot\iota\big(\extPsi_P(y,\aaa,\bbb)\big) = \sum_{\cC} \Poin_{P,\cC}(y) \cdot \wt_\cC(\aaa,\bbb)\ \in \N[y]\langle \aaa,\bbb\rangle \,,
  \]
  where the sum ranges over all chains $\cC = \{\cC_1 < \dots < \cC_{k+1}\}$ starting in $\cC_1 = \hat 0$ and ending in $\cC_{k+1} = \hat 1$ and where~$\iota$ deletes the first variable from every $\aaa\bbb$-monomial.
\end{proposition}
Observe that the factor~$\bbb$ on the left-hand side comes from the assumption that $\cC_1 = \hat 0$ and thus every $\aaa,\bbb$-polynomial $\wt_\cC(\aaa,\bbb)$ starts with the variable~$\bbb$ as $0 = \rank(\cC_1)$.
For later reference, we set
\begin{equation}
  \label{eq:extpsiiota}
  \extPsiIota_P(y,\aaa,\bbb) = \iota\big(\extPsi_P(y,\aaa,\bbb)\big)\,.
\end{equation}

We next present two of the main result of~\cite{poincareextended}, from which we then deduce the proposed evaluation to obtain the Chow and the augmented Chow polynomials.
As in \Cref{sec:main}, we write $\lambda_\cF = (\lambda_1,\dots,\lambda_n)$ for the labeling along a maximal chain~$\cF = \{\cF_0 \prec \dots \prec \cF_n\}$ with $\lambda_i = \lambda(\cF_{i-1} \prec \cF_i)$.
For a subset $E \subseteq \{1,\dots,n\}$, we moreover set $\lambda_{\cF,E} = (\lambda'_0,\lambda'_1,\dots,\lambda'_n)$ with $\lambda'_0 = 0$ and $\lambda'_i = -\lambda_i$ if $i \in E$ and $\lambda'_i = \lambda_i$ if $i \notin E$.
We in addition define $\StatMon_{\cF,E}(\aaa,\bbb) = u_0 \cdots u_{n-1}$ to be the monomial in $\aaa,\bbb$ with
\[
  u_i = \begin{cases}
           \aaa &\text{if } \lambda'_i \leq \lambda'_{i+1}\,, \\
           \bbb &\text{if } \lambda'_i > \lambda'_{i+1}\,, \\
        \end{cases}
\]
so the monomial $\StatMon_{\cF,E}(\aaa,\bbb)$ detects the ascent-descent pattern of~$\lambda_{\cF,E}$.

\begin{theorem}[{\!\!\cite[Theorem~2.7]{poincareextended}}]
\label{thm:combinatorial-interp}
  Let~$P$ be a finite graded poset of rank~$n$, admitting an $R$-labeling~$\lambda$.
  Then 
  \[
    \extPsi_P(y,\aaa,\bbb)
      = \sum_{\cF,E} y^{\#E}\cdot\StatMon_{\cF,E}(\aaa,\bbb)
  \]
  where the sum ranges over all maximal chains~$\cF$ in~$P$ and all subsets~$E \subseteq \{1,\dots,n\}$.
\end{theorem}

\begin{corollary}[{\!\!\cite[Corollary~2.9]{poincareextended}}]
\label{thm:refinement-of-ber}
  Let~$P$ be a finite graded poset of rank~$n$, admitting an $R$-labeling~$\lambda$.
  Then 
  \[
    \extPsi_P(y,\aaa,\bbb) = \omega\big(\Psi_P(\aaa,\bbb)\big)
  \]
  where the substitution~$\omega$ first replaces all occurrences of~$\aaa\bbb$ with
  $\aaa\bbb + y\bbb\aaa + y\aaa\bbb + y^2\bbb\aaa$ and then simultaneously
  replaces all remaining occurrences of~$\aaa$ with $\aaa+y\bbb$ and~$\bbb$ with
  $\bbb+y\aaa$.
\end{corollary}

It is immediate that the $\aaa\bbb$-index in~\eqref{eq:abindex} has the expansion
\begin{equation}
\label{eq:abindexeval}
  \Psi_P(\aaa,\bbb) = \sum_{\cF} \StatMon_{\cF,\emptyset}(\aaa,\bbb)
\end{equation}
and that every monomial $\StatMon_{\cF,\emptyset}(\aaa,\bbb) = u_0 \cdots u_{n-1}$ starts with the variable~$u_0 = \aaa$.

\subsection{Evaluating the Poincaré-extended $\aaa\bbb$-index}
\label{sec:proofs}

The main results of this section are the following evaluations.

\begin{theorem}
\label{thm:main3}
  Let~$P$ be a finite graded poset of rank~$n$ admitting an $R$-labeling.
  Then
  \begin{align*}
    \extPsi_P(-x,1,x)     &= (1-x)^{n} \cdot \CHilbAug_P(x)\,, \\
    \extPsiIota_P(-x,1,x) &= (1-x)^{n} \cdot \CHilb_P(x)\,.
  \end{align*}
\end{theorem}

Before proving this evaluation, we deduce \Cref{thm:main1,,thm:main2}.

\begin{proof}[Proof of \Cref{thm:main1,,thm:main2}]
  Using the expansion~\eqref{eq:abindexeval}, an immediate consequence of \Cref{thm:combinatorial-interp} and \Cref{thm:refinement-of-ber} is that
  \[
    \extPsi_P(-x,1,x) = \sum_{\cF} \omega\big(\StatMon_{\cF,\emptyset}(\aaa,\bbb)\big)\big|_{y = -x,\ \aaa=1,\ \bbb=x}
  \]
  where the sum ranges over all maximal chains~$\cF$ in~$P$.
  Applying the evaluation $y = -x,\ \aaa=1,\ \bbb=x$ to the substitution~$\omega$ gives that every occurrence in $\StatMon_{\cF,\emptyset}$ of the pattern $\aaa\bbb$ becomes a factor $x -x^2 - x^2 + x^3 = x(1-x)^2$, every remaining~$\aaa$ yields a factor $1-x^2 = (1+x)(1-x)$, and every remaining~$\bbb$ yields a factor~$0$.
  We therefore obtain
  \begin{align*}
    \omega\big(\StatMon_{\cF,\emptyset}(\aaa,\bbb)\big)\big|_{y = -x,\ \aaa=1,\ \bbb=x} &= 0
    \intertext{
      if $\Des(\cF)$ is not isolated, and otherwise
    }
    \omega\big(\StatMon_{\cF,\emptyset}(\aaa,\bbb)\big)\big|_{y = -x,\ \aaa=1,\ \bbb=x}
    &= (1-x)^n\cdot x^{\des(\cF)}(1+x)^{n-2\ \!\!\des(\cF)}\,.
  \end{align*}
  Summing over all maximal chains~$\cF$ in~$P$ such that for $\Des(\cF)$ is isolated then gives
  \begin{align}
    \extPsi_P(-x,1,x) =  (1-x)^n\sum_\cF x^{\des(\cF)} (x+1)^{n-2\ \!\! \des(\cF)}\,.
    \label{eq:chow-expansion}
  \end{align}
  This proves \Cref{thm:main2}.

  \medskip

  The argument to deduce \Cref{thm:main1} is almost the same.
  The only difference is that now every initial variable is removed from every $\aaa\bbb$-monomial to obtain $\extPsiIota_P(y,\aaa,\bbb)$ from $\extPsi_P(y,\aaa,\bbb)$.
  This is, we have
  \[
    \extPsiIota_P(y,\aaa,\bbb)
    = \iota\big(\extPsi_P(y,\aaa,\bbb)\big)
    = \sum_{\cF} \iota\big(\omega\big(\StatMon_{\cF,\emptyset}(\aaa,\bbb)\big)\big)\,.
  \]
  We consider the two cases whether $\StatMon_{\cF,\emptyset}(\aaa,\bbb)$ starts with~$\aaa\aaa$ or with~$\aaa\bbb$.
  If it starts with~$\aaa\aaa$, we obtain
  \[
    \iota\big(\omega\big(\StatMon_{\cF,\emptyset}(\aaa,\bbb)\big)\big) =
    (1+y)\omega\big(\iota\big(\StatMon_{\cF,\emptyset}(\aaa,\bbb)\big)\big)
  \]
  as the first variable~$\aaa$ is evaluated to $\iota\omega(\aaa) = \iota(\aaa + y\bbb) = 1+y$.
  If it starts with $\aaa\bbb$, we obtain
  \[
    \iota\omega\big(\StatMon_{\cF,\emptyset}(\aaa,\bbb)\big)\big|_{y = -x,\ \aaa=1,\ \bbb=x} = 0
  \]
  as the initial two variables $\aaa\bbb$ are evaluated to
  \[
    \iota\omega(\aaa\bbb) = \iota((1+y)\aaa\bbb + y(1+y)\bbb\aaa) = (1+y)\bbb + y(1+y)\aaa
  \]
  and
  \[
    (1+y)\bbb + y(1+y)\aaa\big|_{y = -x,\ \aaa=1,\ \bbb=x} = 0
  \]
  Summing now over all maximal chains~$\cF$ in~$P$ such that $\Des(\cF)$ is isolated and $1 \notin \Des(\cF)$ then gives
  \begin{align*}
    \extPsiIota_P(-x,1,x) = (1-x)^n \sum_\cF x^{\des(\cF)} (x+1)^{n-1-2\ \!\! \des(\cF)}\,.
  \end{align*}
  This proves \Cref{thm:main1}.
\end{proof}

\begin{proof}[Proof of \Cref{thm:main3}]
  We evaluate the Chow polynomial and the extended $\aaa\bbb$-index as follows, where the sums range over all chains $\cC = \{\cC_1 < \dots < \cC_k < \cC_{k+1}\}$ starting in $\cC_1 = \hat 0$ and ending in $\cC_{k+1} = \hat 1$.
  We write $k_\cC = k$ to denote its length.
  Then
  \begin{align*}
    \CHilb_P(x)
    &= \sum_{\cC} \redChi_{P,\cC}(x)
    = \sum_{\cC} \prod_{i=1}^{k_\cC} \redChi_{[\cC_i,\cC_{i+1}]}(x) \\
    &= \sum_{\cC} \prod_{i=1}^{k_\cC} x^{\rank(\cC_{i+1})-\rank(\cC_i)}(x-1)^{-1}\cdot \Poin_{[\cC_i,\cC_{i+1}]}(-x^{-1}) \\
    &= x^n\sum_{\cC} \prod_{i=1}^{k_\cC} (x-1)^{-1}\cdot \Poin_{[\cC_i,\cC_{i+1}]}(-x^{-1}) \\
    &= x^n\sum_{\cC} (x-1)^{-k_\cC} \Poin_{P,\cC}(-x^{-1}) \\
    &= x^n(x-1)^{-1} \extPsiIota_P(-x^{-1},\tfrac{x}{x-1},\tfrac{1}{x-1})
  \end{align*}
  where we used in the last equation that $\wt_{\cC}(\aaa,\bbb)$ has exactly~$n-1$ factors, either being~$\bbb$ for ranks in~$\cC \setminus \{\cC_{k+1}\}$ or being $\aaa-\bbb$ for ranks not in~$\cC$, and moreover that evaluating~$\bbb\mapsto \frac{1}{x-1}$ and $\aaa \mapsto \frac{x}{x-1}$ evaluates $\aaa-\bbb \mapsto 1$ and at the same time yields a factor $(x-1)^{1-k_\cC}$.
  Using the symmetry $\CHilb_P(x) = x^{n-1}\CHilb_P(x^{-1})$, we obtain
  \begin{align}
    \CHilb_P(x)
    &= x^{-1}(x^{-1}-1)^{-1} \extPsiIota_P(-x,\tfrac{1}{1-x},\tfrac{x}{1-x}) \notag \\
    &= (1-x)^{-n} \extPsiIota_P(-x,1,x) \label{eq:chow-evaluation2}
  \end{align}
  where we used that $\extPsiIota_P(y,\aaa,\bbb)$ is homogeneous of $\aaa\bbb$-degree $n-1$ and thus
  \[
    \extPsiIota_P(-x,\tfrac{1}{1-x},\tfrac{x}{1-x}) = (1-x)^{-(n-1)}\cdot\extPsiIota_P(-x,1,x)\,.
  \]
  The argument for
  \[
    \CHilbAug_P(x) = (1-x)^{-n}\extPsi_P(-x,1,x)
  \]
  is word for word the same with the only difference that chains $\cC = \{\cC_1 < \dots < \cC_k < \cC_{k+1}\}$ do not necessarily need to start in $\cC_1 = \hat 0$, and thus the factor $x^{\rank(\cC_1)}$ in the definition of the augmented Chow polynomial ensures that $x^n = x^{\rank(\cC_1)}\cdot \prod_{i=1}^{k_\cC} x^{\rank(\cC_{i+1})-\rank(\cC_i)}$.
\end{proof}

\begin{remark}
  In the proof, we used the symmetry $\CHilb_P(x) = x^{n-1}\CHilb_P(x^{-1})$.
  As the symmetry of $(1-x)^{-n}\extPsi_P(-x,1,x)$ is clear from its combinatorial description established in~\eqref{eq:chow-expansion}, this is not a required property of the Chow polynomial as we could instead have used the symmetry of the right-hand side of~\eqref{eq:chow-evaluation2}.
  We also observe that we have not made use of the property that the poset admits an $R$-labeling.
  This evaluation thus holds for all finite graded posets.
\end{remark}

\subsection{Evaluating the coarse flag Hilbert-Poincaré series}

The coarse flag Hilbert-Poincaré series~$\cfHP_P(y,t)$ is originally defined and studied in~\cite{MaglioneVoll} as an algebraic invariant in geometric group theory.
It is shown in~\cite[Corollary~2.22]{poincareextended} that this rational function can be expressed as
  \[
    \cfHP_P(y,t) = (1-t)^{-n}\extPsiIota_P(y,1,t)\,.
  \]

We thus immediately obtain the following specialization from \Cref{thm:main3}.

\begin{corollary}
\label{cor:coarsflagHP}
  Let~$P$ be a finite graded poset of rank~$n$ admitting an $R$-labeling.
  The Chow polynomial is then obtained from the coarse flag Hilbert-Poincaré series by the specialization
  \begin{align*}
    \CHilb_P(x) = \cfHP_P(-x,x)\,.
  \end{align*}
\end{corollary}

Applying~$~\iota$ to $\extPsi_P(y,\aaa,\bbb)$ restricts the situation to chains starting in~$\hat 0$ as in \Cref{prop:iotaapp}.
It is thus not immediately possible to also obtain the augmented Chow polynomial from the coarse flag Hilbert-Poincaré series.

\section{Chow and augmented Chow polynomials\\ for braid arrangement}
\label{sec:typeA}

In this section, we apply the main theorems to obtain explicit closed formulas of the Chow and the augmented Chow polynomials for the \Dfn{braid arrangement}~$\cB_n$  given by the hyperplanes
\[
  H_{ij} = \set{ x_i = x_j}{ 1 \leq i < j \leq n+1 }
\]
in $V = \big\{ (x_1,\dots,x_{n+1}) \in \RR^{n+1} \mid x_1+\dots+x_{n+1} = 0\big\} \cong \RR^n$.
Inductive formulas for all reflection arrangements of classical types have been established in~\cite{zbMATH06026123}.

\begin{theorem}
\label{thm:typeA}
  We have
  \begin{align*}
     \CHilb_{\cB_n}(x)
     &= \sum_{(a_1,\dots,a_n)} a_1\cdot\ldots\cdot a_n\cdot x^{\asc(a_1,\dots,a_n)} (x+1)^{n-1-2\ \!\!\asc(a_1,\dots,a_n)}\,,
  \end{align*}
   where the sum ranges over all $a = (a_1,\dots,a_n)$ with $a_i \in \{1,\dots,n+1-i\}$ such that $\Des(a)$ is isolated and $1 \notin \Des(a)$.
  We moreover have
  \begin{align*}
     \CHilbAug_{\cB_n}(x)
     &= \sum_{(a_1,\dots,a_n)} a_1\cdot\ldots\cdot a_n\cdot x^{\asc(a_1,\dots,a_n)} (x+1)^{n-2\ \!\!\asc(a_1,\dots,a_n)}\,,
  \end{align*}
   where the sum ranges over all $a = (a_1,\dots,a_n)$  with $a_i \in \{1,\dots,n+1-i\}$ such that $\Des(a)$ is isolated.
\end{theorem}

\begin{example}
  Let $n=2$.
  For the Chow polynomials, the sum ranges over $\{(1,1)\}$ and we obtain
  \begin{align*}
     \CHilb_{\cB_2}(x) &= 1 \cdot 1 \cdot x^0(x+1)^{2-1-0} = x+1
  \end{align*}
  and for the augmented Chow polynomials, the sum ranges over $\{(1,1), (2,1)\}$ and we obtain
  \begin{align*}
     \CHilbAug_{\cB_2}(x)
     &= 1 \cdot 1 \cdot(x+1)^2 + 1 \cdot 2 \cdot x^1
     = x^2 + 4x + 1\,.
  \end{align*}
  Let now $n=3$.
  For the Chow polynomials, the sum ranges over $\{(1,1,1)$, $(1,2,1)$, $(2,2,1)\}$ and we obtain
  \begin{align*}
     \CHilb_{\cB_3}(x) &= 1 \cdot (x+1)^2 + 2\cdot x + 4\cdot x = x^2+8x+1
  \end{align*}
  and for the augmented Chow polynomials, the sum ranges over $\{(1,1,1)$, $(1,2,1)$, $(2,2,1)$, $(2,1,1)$, $(3,1,1)\}$ and we obtain
  \begin{align*}
     \CHilbAug_{\cB_3}(x)
     &= 1\cdot(x+1)^3 + 2\cdot x(x+1) + 4\cdot x(x+1) + 2\cdot x(x+1) + 3\cdot x(x+1) \\
     &= (x+1)^3 + 11\cdot x(x+1) \\
     &= x^3 + 14x^2 + 14x + 1\,.
  \end{align*}
  These calculations agree with the results presented in~\cite[Section~3.2]{zbMATH06026123}.
\end{example}

We prove \Cref{thm:typeA} by analyzing the maximal chains in the lattice of flats of the braid arrangement for one particular $R$-labeling.
The lattice of flats of~$\cB_n$ is the \Dfn{lattice of set partitions} of~$\{1,\dots,n+1\}$ of rank~$n$, with cover relations being given by merging two blocks, denoted by~$\Pi_n$.
Its minimal element is the all-singletons set partition and its maximal element is the all-in-one-block set partition.
It moreover has an $R$-labeling~$\lambda$ given by labeling a cover relation in which the two blocks~$B$ and~$B'$ are merged by the \Dfn{max-of-min labeling} $\max\big\{\!\min(B),\min(B')\big\}$, see~\cite[Example~2.9]{MR0570784}.
Observe that the max-of-min labeling along any maximal chain is a permutation of $\{2,\dots,n+1\}$.

\begin{remark}
  This $R$-labeling was recently used in~\cite{hoster2024coarseflaghilbertpoincareseries} to derive an explicit form for the coarse flag Hilbert-Poincaré series of the braid arrangement.
  In light of \Cref{cor:coarsflagHP}, it seems natural that the max-of-min labeling also plays a crucial role in the understanding of the Chow polynomials of the braid arrangements.
  Moreover, Brändén and Vecchi recently proved the real-rootedness of the augmented Chow polynomial of the braid arrangement in~\cite{brandenvecchi}.
\end{remark}

To deduce \Cref{thm:typeA} from \Cref{thm:main1,,thm:main2}, we calculate for a given permutation~$\sigma$ of $\{2,\dots,n+1\}$ the number of maximal chains~$\cL$ in~$\Pi_n$ for which $\lambda_\cF = \sigma$.
(Our tailor-made version of) the \Dfn{inversion sequence} $\Inv(\sigma) = (a_1,\dots,a_n)$ of a permutation~$\sigma = [\sigma_1,\dots,\sigma_n]$ of $\{2,\dots,n+1\}$ in one-line notation is given by
\[
  a_i = \#\{ j \geq i \mid \sigma_j \leq \sigma_i \}\,.
\]
Observe that the entry~$a_i$ in the inversion sequence~$\Inv(\sigma)$ may be computed from the partial permutation~$[\sigma_1,\dots,\sigma_i]$ as
\begin{equation}
\label{eq:partialinv}
  a_i = \sigma_i - \#\{ j \leq i \mid \sigma_j \leq \sigma_i \}\,.
\end{equation}

\begin{example}
  Let $n = 3$.
  Then
  \begin{center}
    \begin{tabular}{c|c}
    $\sigma$ & $\Inv(\sigma)$ \\ \hline
    $[2,3,4]$ & $(1,1,1)$ \\
    $[2,4,3]$ & $(1,2,1)$ \\
    $[3,2,4]$ & $(2,1,1)$ \\
    $[3,4,2]$ & $(2,2,1)$ \\
    $[4,2,3]$ & $(3,1,1)$ \\
    $[4,3,2]$ & $(3,2,1)$ \\
    \end{tabular}
  \end{center}
\end{example}
The map $\sigma \mapsto \Inv(\sigma)$ is moreover a bijection between permutations of~$\{2,\dots,n+1\}$ and sequences $(a_1,\dots,a_n)$ for which $a_i \in \{1,\dots,n+1-i\}$.

\begin{proposition}
\label{prop:countingmaxchains}
  Let~$\sigma$ be a permutation of $\{2,\dots,n+1\}$ with inversion sequence $\Inv(\sigma) = (a_1,\dots,a_n)$.
  Then
  \[
    \#\{ \text{maximal chains } \cF \text{ in } \Pi_n \mid \lambda_\cF = \sigma \} = a_1\cdot\ldots\cdot a_n\,.
  \]
\end{proposition}

\begin{proof}
  Given a set partition $\mathbf B = B_1|\dots|B_\ell \in \Pi_n$ and a maximal chain $\cF$ going through~$\mathbf B$, the max-of-min labels among the cover relations in~$\cF$ before reaching~$\mathbf B$ are those indices that have already appeared as max-of-min's.
  These are
  \[
    \{2,\dots,n+1\} \setminus \{ \min B_1,\dots, \min B_\ell \}\,.
  \]
  Conversely, the labels after reaching~$\mathbf B$ are $\{ \min B_1,\dots, \min B_\ell \} \setminus \{1\}$.
  Moreover, the labeled interval between $\mathbf B$ and the maximal element $\{1,\dots,n+1\} \in \Pi_n$ does not depend on~$\mathbf B$ itself as it is simply the set partition lattice on the set $\{ \min B_1,\dots, \min B_\ell \}$ again with the max-of-min labeling.

  Let now $i \in \{1,\dots,n\}$, let $\sigma = [\sigma_1,\dots,\sigma_i]$ be a non-empty partial permutation, \ie, a list of~$i$ distinct numbers in $\{2,\dots,n+1\}$ and let $\sigma' = [\sigma_1,\dots,\sigma_{i-1}]$ be the partial permutation obtained by removing from~$\sigma$ the last entry.
  Set~$a_i$ to be the $i$'s entry in the inversion sequence of any permutation starting with~$\sigma$ as given in~\eqref{eq:partialinv}.
  Finally, denote by~$m_\sigma$ the number of maximal chains~$\cF$ in~$\Pi_n$ for which the labeling~$\lambda_\cF = (\lambda_1,\dots,\lambda_n)$ starts with $(\lambda_1,\dots,\lambda_i) = (\sigma_1,\dots,\sigma_i)$, meaning that $m_\sigma$ counts initial chains $\cC = (\cC_0,\dots,\cC_i)$ such that $\rank(\cC_j) = j$ and $\lambda_j = \sigma_j$ for all $j \in \{1,\dots,i\}$.
  Similarly, denote by~$m_{\sigma'}$ the number of maximal chains for which the labeling starts with $(\sigma_1,\dots,\sigma_{i-1})$.
  Then
  \begin{equation*}
    m_{\sigma'} = m_\sigma\cdot a_i\,,
  \end{equation*}
  because there are exactly~$a_i$ many blocks~$B$ into which one can merge the block containing~$\sigma_i$, say~$B'$, such that $\sigma_i = \max\{\min B,\min B'\}$.
  The statement follows from recursively applying this counting formula.
\end{proof}

\begin{proof}[Proof of \Cref{thm:typeA}]
  Let~$\sigma = [\sigma_1,\dots,\sigma_n]$ be a permutation of $\{2,\dots,n+1\}$ with inversion sequence~$\Inv(\sigma) = (a_1,\dots,a_n)$.
  Then $\sigma_i > \sigma_{i+1}$ is a descent of~$\sigma$ if and only if $a_i > a_{i+1}$ is a descent of the inversion sequence.
  Both parts of the statement follow by using \Cref{prop:countingmaxchains} to simplify the ranges of the sums in \Cref{thm:main1,,thm:main2}.
\end{proof}

\bibliographystyle{alpha}
\bibliography{bibliography}

\end{document}